\documentclass[12pt]{amsart}
\usepackage{amscd,amsmath,amssymb,amsfonts}
\usepackage[cmtip, all]{xy}

\newtheorem{thm}{Theorem}[section]
\newtheorem{prop}[thm]{Proposition}
\newtheorem{lem}[thm]{Lemma}
\newtheorem{cor}[thm]{Corollary}

\theoremstyle{definition}

\newtheorem{defn}[thm]{Definition}
\newtheorem{remk}[thm]{Remark}
\newtheorem{remks}[thm]{Remarks}

\newtheorem{exm}[thm]{Example}
\newtheorem{exms}[thm]{Examples}
\newtheorem{notat}[thm]{Notation}

\numberwithin{equation}{section}

{\hfill$\square$\end{defn}}

{\hfill$\square$\end{remk}}

{\hfill$\square$\end{remks}}

{\hfill$\square$\end{exm}}

{\hfill$\square$\end{exms}}

{\hfill$\square$\end{notat}}

\newcommand{\sV}{{\mathcal V}}

\newcommand{\A}{{\mathbb A}}

\newcommand{\C}{{\mathbb C}}

\newcommand{\bL}{{\mathbb L}}

\newcommand{\N}{{\mathbb N}}

\newcommand{\Q}{{\mathbb Q}}

\newcommand{\Z}{{\mathbb Z}}

\newcommand{\surj}{\twoheadrightarrow}
\newcommand{\inj}{\hookrightarrow}

\newcommand{\Sym}{{\operatorname{\rm Sym}}}

\newcommand{\ds}{{/\kern-3pt/}}

\newcommand{\h}{{\operatorname{\rm H}}}
\newcommand{\ch}{{\operatorname{\rm CH}}}

\newcommand{\un}{\underline}
\newcommand{\ov}{\overline}
\newcommand{\wt}{\widetilde}
\newcommand{\wh}{\widehat}

\renewcommand{\dim}{\text{\rm dim}}

\newcommand{\tuborg}{\left\{\begin{array}{ll}}
\newcommand{\sluttuborg}{\end{array}\right.}

\begin{document}
\title{Cobordism of flag bundles}
\author{Amalendu Krishna}
\address{School of Mathematics, Tata Institute of Fundamental Research,  
1 Homi Bhabha Road, Colaba, Mumbai, India}
\email{amal@math.tifr.res.in}

\baselineskip=10pt 
  

%\maketitle
\begin{abstract}
Let $G$ be a connected linear algebraic group over a field $k$ of
characteristic zero and let $P$ be a parabolic subgroup of $G$ containing 
a fixed maximal torus $T$. For a scheme $X$ of finite type over $k$
and a principal $G$-bundle $E \to X$, we describe the rational algebraic 
cobordism of the flag bundle $E/P \to X$ in terms of the cobordism
groups of $X$ and the classifying space $BT$. In particular, we obtain
formulae for the algebraic cobordism groups of the various flag 
bundles associated to a vector bundle on a scheme. 
As a consequence, we describe the cobordism group of any 
principal bundle over a scheme. We also obtain similar formula for the
higher Chow groups of flag bundles.
%Our description is 
%analogous to the similar formula for the Chow groups of flag bundles 
%obtained earlier by Vistoli. 
\end{abstract}

\maketitle
%\tableofcontents

\section{Introduction}
Let $k$ be a field of characteristic zero. In this paper, we shall consider
only those schemes which are quasi-projective over $k$. 
Based on the construction of the motivic algebraic cobordism spectrum
$MGL$ by Voevodsky in the stable homotopy category of $k$, and the
already known cobordism theory for complex manifolds
\cite{Quillen}, Levine and Morel \cite{LM} invented the algebraic cobordism 
theory $\Omega_*(-)$. The most important aspect of this theory is that
$\Omega_*(-)$ is the universal oriented Borel-Moore homology theory in the
category of $k$-schemes. In particular, it is the universal oriented
cohomology theory in the category of smooth schemes over $k$. 

As a consequence, many known theories, e.g., 
algebraic $K$-theory, Chow groups, can be directly obtained from the 
cobordism theory of Levine and Morel. This makes the question of 
describing the algebraic cobordism groups of various schemes interesting 
and important. Since this theory has been invented only some years ago, not 
many cases of computations of $\Omega_*(-)$ have been known. Levine and Morel
showed that the coefficient ring $\Omega^*(k)$ is isomorphic to the known
Lazard ring. They were also able to describe the algebraic cobordism of
a projective bundle in terms of the cobordism group of the base scheme.
The principal aim of this paper is to generalize this description to the case
of arbitrary flag bundles. As a consequence, we also describe the
cobordism groups of principal $G$-bundles over $k$-schemes. 

In order to motivate our main results, we recall the following 
result, due to Borel and Leray, well known in algebraic topology and its 
analogue in algebraic geometry, due to Vistoli \cite{Vistoli}. Assume 
$k = \C$ is the
field of complex numbers and let $G$ be a connected and reductive complex
algebraic group. We fix a maximal torus $T$ of $G$, a Borel
subgroup $B$ of $G$ containing $T$ and let $W$ denote the Weyl group of $G$
with respect to $T$. Let $\widehat{T}$ denote the character group of $T$
and let ${\rm Sym}(\widehat{T})$ denote the symmetric algebra of 
$\widehat{T} \otimes \Q$ over $\Q$. Let $X$ be a complex manifold and let 
$E \to X$ be a principal 
$G$-bundle. The reader can think of it as a $G(\C)$-fiber bundle over $X(\C)$.
Since the principal bundles are represented by the maps to the classifying 
spaces in topology, we immediately get the characteristic homomorphisms
\[
{{\Sym}(\wh{T})}^W \cong \h^*(BG, \Q) \xrightarrow{c_E} \h^*(X, \Q) \ 
{\rm and} \ {\rm Sym}(\widehat{T}) \xrightarrow{\alpha_E} \h^*(E/B, \Q).
\]
The homomorphism $\alpha_E$ sends a character $t$ of $T$ to the first Chern
class of the associated line bundle $E/B \stackrel{T, t}{\times} \A^1 \to E/B$.
This induces a homomorphism of ${\rm Sym}(\widehat{T})$-algebras
\begin{equation}\label{eqn:fiber-bundle}
\h^*(X, \Q) {\otimes}_{{{\rm Sym}(\widehat{T})}^W} {\rm Sym}(\widehat{T})
\xrightarrow{\lambda_X} \h^*(E/B, \Q)
\end{equation}
and is an isomorphism.     
 
Let $X$ be now a scheme and let $\ch_*(X)$ denote the rational Chow group
of algebraic cycles on $X$ modulo the rational equivalence. Let $A^*(X)$
be the Fulton-MacPherson bivariant cohomology ring of $X$. Recall that
this is a subring of the endomorphism ring of $\ch_*(X)$. In the same set up
as above, Vistoli \cite{Vistoli} showed that there are still the characteristic
maps 
\[
{{\Sym}(\wh{T})}^W  \xrightarrow{c_E} A^*(X) \ 
{\rm and} \ {\rm Sym}(\widehat{T}) \xrightarrow{\alpha_E} A^*(E/B)
\]
such that the induced map of ${\Sym}(\wh{T})$-modules
\begin{equation}\label{eqn:Chow-bundle}
\ch_*(X) {\otimes}_{{{\rm Sym}(\widehat{T})}^W} {\rm Sym}(\widehat{T})
\xrightarrow{\lambda_X} \ch_*(E/B)
\end{equation}
is an isomorphism. This completely describes the Chow groups of the flag
bundle in terms of the Chow group of the base. 

In this paper, we study similar questions for the description of the
higher Chow groups and more importantly, the algebraic 
cobordism groups of generalized flag bundles over any base scheme.
The similar techniques can also be used to write down the description of
the complex cobordism of flag bundles over complex manifolds.
In case of the higher Chow groups, we obtain a more
direct proof of the above formula using the localization sequence for these 
groups. In particular, this yields a different and simpler proof of Vistoli's 
theorem for the Chow groups. The proof in the case of the cobordism
becomes much more complicated, mainly due to the absence 
of the {\sl higher cobordism groups} at present. In this case, we adapt
some of the arguments of \cite{Vistoli} to the case of cobordism. We now
state our main results.

Let $G$ be a connected linear algebraic group and let $T$, $B$ and $W$ be
a fixed split maximal torus, a Borel subgroup containing the maximal torus and
the associated Weyl group respectively. We shall often denote this datum
by the quadruplet $(G, T, B, W)$. Let $r$ denote the rank of $T$.
Let $P$ be a parabolic subgroup of $G$ containing $B$ and let $W_P$ denote
the Weyl group of the Levi subgroup of $P$ with respect to $T$.

Let $S(G)$ and $C(G)$ denote the $G$-equivariant rational Chow ring 
and the algebraic cobordism ring of ${\rm Spec}(k)$ 
(see Section~\ref{section:prelim} below). Since we are interested in
describing the higher Chow groups and cobordism groups with the rational
coefficients, we make the convention that an abelian group $A$ for us
will actually mean $A \otimes_{\Z} \Q$. Furthermore, we shall write
the higher Chow groups and the cobordism groups cohomologically in this paper
in the sense that $\ch^i(X, n)$ and $\Omega^i(X)$ will mean the groups
$\ch_{{\rm dim}(X)-i}(X, n)$ and $\Omega_{{\rm dim}(X)-i}(X)$ respectively. 
For any scheme $X$, we shall write the full Chow groups as
\[
\ch^*(X) = \stackrel{\infty}{\underset{i = 0}\oplus} \
\stackrel{\infty}{\underset{n = 0}\oplus} \ch^i(X,n).
\]

Let $p: E \to X$ be a principal $G$-bundle and let $\pi : E/P \to X$ be the
flag bundle associated to the parabolic subgroup $P$. 
We show in Section~\ref{section:prelim} below
that the algebraic cobordism groups $\Omega^*(X)$ and $\Omega^*(E/P)$ are
modules over the rings $C(G)$ and $C(P)$ respectively.
Moreover, it is known ({\sl cf.} \cite[Theorem~6.6]{Krishna3}) that
$C(G) = {C(T)}^W$ and $C(P) \cong {C(T)}^{W_P}$.
The similar methods also show that  
the higher Chow groups $\ch^*(X, n)$ and $\ch^*(E/P, n)$ are 
modules over the rings $S(G) = S^W$ and $S(P) = S^{W_P}$ respectively. 
Now we have:
\begin{thm}\label{thm:Cob-PR}
The natural map of $C(P)$-modules
\begin{equation}\label{eqn:CM1}
\lambda_X : \Omega^*(X) \otimes_{C(G)} C(P) \to \Omega^*(E/P)
\end{equation}
is an isomorphism. Moreover, it is an isomorphism of rings if $X$ is smooth.
\end{thm}

\begin{thm}\label{thm:HC-Main}
The natural map of $S(P)$-modules
\begin{equation}\label{eqn:CM2}
\alpha_X : \ch^*(X) \otimes_{S(G)} S(P) \to \ch^*(E/P)
\end{equation}
is an isomorphism. This is an isomorphism of rings if $X$ is smooth.
\end{thm} 

As consequences of these results, we obtain the formulae ({\sl cf.} 
Corollaries~\ref{cor:Cob-Group} and ~\ref{cor:HC-Main-group}) for the 
cobordism and the higher Chow groups of principal bundles.
We remark here that as we are working with the rational coefficients, the
assumption about the maximal torus $T$ being split is not a necessary one.
One can reduce to this case by the transfer arguments.

We conclude the introduction with a brief outline of the contents of this
paper. We recall the definitions and some important properties of the
ordinary and the equivariant algebraic cobordism in the next section.
We use these fundamental properties to construct our map $\lambda_X$
in Section~\ref{section:prelim}. We also deduce some functorial properties
of this map with respect to morphisms between schemes.
In section~\ref{section:ALG-R}, we prove some algebraic results and use these
together with some results of \cite{Krishna3} to deduce our main result
for algebraic cobordisms of trivial flag bundles over smooth schemes.
In Section~\ref{section:Surj}, we prove the surjectivity of $\lambda_X$
using the localization sequence for the cobordism, the corresponding
result for the trivial flag bundles and an induction argument.
We prove Theorem~\ref{thm:Cob-Main} in Section~\ref{section:ProofI} by first 
proving it for the trivial flag bundles, which uses the proof of the similar 
result for the Chow groups of trivial bundles in \cite{Vistoli}, and then 
using a filtration argument for general schemes. The final proof of 
Theorem~\ref{thm:Cob-PR} is given in Section~\ref{section:Parabolic}, where
we deduce this from the case of flag bundles associated to the Borel
subgroups. The last section is devoted to the proof of 
Theorem~\ref{thm:HC-Main} where the main tool is the long exact localization
sequence for the higher Chow groups.
  
\section{Recollection of ordinary and equivariant cobordism}
\label{section:prelim1}
In this section, we briefly recall the definitions and basic properties
of the ordinary and equivariant algebraic cobordism.
\subsection{Algebraic cobordism}\label{subsection:Alg-C}
Recall from \cite{LM} that for any scheme $X$ and any
$i \in \Z$, the algebraic cobordism group $\Omega^i(X)$ is given by the
quotient of the $\Q$-vector space $Z^i(X)$ on the classes of projective 
morphisms $[Y \xrightarrow{f} X]$, where $Y$ is a smooth scheme and $f$ has
relative codimension $i = {\dim}(X) - {\rm dim}(Y)$. 
This quotient is obtained by the relations in 
$Z^i(X)$ defined by certain axioms like the dimension axiom, section axiom and
the formal group law. It was later shown by Levine and Pandharipande 
\cite{LP} that $\Omega^i(X)$ can also be described as the quotient of
$Z^i(X)$ by the subspace generated by those cobordism cycles which are given
by the {\sl double point degeneration} relation. In particular, there
is a natural surjection $Z^*(X) \surj \Omega^*(X)$. It also follows that
$\Omega^*(X)$ is a graded $\Q$-vector space, where the grading is given
by the codimension of a cobordism cycle. Moreover, $\Omega^i(X) = 0$ for 
$i > {\rm dim}(X)$ and $\Omega^i(X)$ could be non-zero for any $- \infty <
i \le {\rm dim}(X)$. In fact, the exterior product on cobordism makes
$\Omega^*(X)$ a graded ring for smooth $X$, which is a graded 
$\Omega^*(k)$-algebra. In general, $\Omega^*(X)$ is a graded 
$\Omega^*(k)$-module.
 
The following is the main result of Levine and Morel from which most of their
other results on algebraic cobordism are deduced. 
\begin{thm}\label{thm:Levine-M}
The functor $X \mapsto \Omega_*(X)$ is the universal Borel-Moore homology
on the category of $k$-schemes. In other words, it is universal among the 
homology theories on this category which have functorial push-forward for 
projective 
morphism, pull-back for smooth morphism (any morphism of smooth schemes), 
Chern classes for line bundles, and which satisfy Projective bundle formula, 
homotopy invariance, the above dimension, section and formal group law axioms.
Moreover, for a $k$-scheme $X$ and closed subscheme $Z$ of $X$ of pure
codimension $p$ with open complement $U$, there is a localization
exact sequence
\[
\Omega_{*}(Z) \to \Omega_*(X) \to \Omega_*(U) \to 0.
\]
\end{thm}
It was also shown in {\sl loc. cit.} that the natural composite map
\[
\Phi : \bL \to \bL \otimes_{\Q} \un{\Omega}^*(k) \surj \Omega^*(k)
\]
\[
a \mapsto [a]
\]
is an isomorphism of commutative graded rings.
Here, $\bL$ is the Lazard ring which is 
a polynomial ring over $\Q$ on infinite but countably many variables and 
is given by the quotient of the polynomial ring $\Q[A_{ij}| (i,j) \in \N^2]$ 
by the relations, which uniquely define the universal formal group law 
$F_{\bL}$ of rank one on $\bL$.  This formal group law is given by the power 
series
\[
F_{\bL}(u,v) = u + v + {\underset {i,j \ge 1} \sum} a_{ij} u^iv^j,
\]
where $a_{ij}$ is the equivalence class of $A_{ij}$ in the ring $\bL$.
The Lazard ring is graded by putting the degree of $a_{ij}$ to be $1-i-j$.
In particular, one has $\bL_{0} = \Q, \bL_{-1} = \Q a_{11}$ and $\bL_{i} = 0$
for $i \ge 1$, that is, $\bL$ is non-positively graded. We refer to 
{\sl loc. cit.} for more properties of algebraic cobordism.

\subsection{Equivariant algebraic cobordism}\label{subsection:Equiv-C}
Let $(G, T, B, W)$ be the datum as above for a given connected linear algebraic
group $G$ over $k$. For a scheme $X$ with a linear action of $G$, the
equivariant algebraic cobordism of $X$ was defined by Deshpande \cite{DP}
when $X$ is smooth and this was later defined and studied for all schemes
in \cite{Krishna3}. Since this is a new theory and since we shall have need for
this here, albeit in a mild way, we briefly recall it.
  
For any integer $j \ge 0$, let $V_j$ be an $l$-dimensional representation of 
$G$ and let $U_j$ 
be a $G$-invariant open subset of $V_j$ such that the codimension of the 
complement $(V_j-U_j)$ in $V_j$ is at least $j$ and $G$ acts freely on 
$U_j$ such that the quotient ${U_j}/G$ is a quasi-projective scheme. Such a 
pair $\left(V_j,U_j\right)$ is called a {\sl good} pair for the $G$-action
corresponding to $j$. It is known that in our set up, 
good pairs always exist ({\sl cf.} \cite[Lemma~9]{EG}). 
Let $X_G$ denote the mixed quotient 
$X \stackrel{G} {\times} U_j$ of the product $X \times U_j$ by the 
diagonal action of $G$, which is free.

Let $X$ be a $k$-scheme of dimension $d$ with a $G$-action.
Fix $j \ge 0$ and let $(V_j, U_j)$ be an $l$-dimensional good pair 
corresponding to $j$. Put 
\begin{equation}\label{eqn:E-cob*}
{\Omega^G_i(X)}_j =  \frac{\Omega_{i+l-g}\left({X\stackrel{G} 
{\times} U_j}\right)}
{F_{d+l-g-j}\Omega_{i+l-g}\left({X\stackrel{G} {\times} U_j}\right)}.
\end{equation}
Here, $F_p\Omega_i(X)$ is the $p$th level of the Niveau filtration 
({\sl cf.} \cite[Section~3]{Krishna3}) which is roughly given by the 
subspace of $\Omega_i(X)$ generated by the images of 
$\Omega_i(Z) \to \Omega_i(X)$
under the push-forward map, where $Z \inj X$ is a closed subscheme of
dimension at most $i$. It is known that ${\Omega^G_i(X)}_j$ is independent of 
the choice of the good pair $(V_j, U_j)$ and one defines
\begin{equation}\label{eqn:ECob-D}
\Omega^G_i(X) : = {\underset {j} \varprojlim} \ \Omega^G_i(X)_j.
\end{equation}
The reader should note from the above definition that unlike the ordinary
cobordism, the equivariant algebraic cobordism $\Omega^G_i(X)$ can be 
non-zero for any $i \in \Z$. We let 
\[
\Omega^G_*(X) = {\underset{i \in \Z} \bigoplus} \ \Omega^G_i(X)
\]
and we let 
\begin{equation}\label{eqn:EQC1}
\Omega^*_G(X) = {\underset{i \in \Z} \bigoplus} \ \Omega^i_G(X),
\ {\rm where} \ \Omega^i_G(X) = \Omega^G_{{\rm dim}(X) -i}(X).
\end{equation}

The equivariant cobordism satisfies all those properties which are listed in
Theorem~\ref{thm:Levine-M} for the ordinary algebraic cobordism. Since we shall
need some of these properties, we state them below for the sake of
completeness.
\begin{thm} ({\sl cf.} \cite[Theorems~5.1, 5.4]{Krishna3})\label{thm:Basic}
The equivariant algebraic cobordism satisfies the following properties. \\
$(i)$ {\sl Functoriality :} The assignment $X \mapsto \Omega_*(X)$ is
covariant for projective maps and contravariant for smooth maps of
$G$-schemes. It is also contravariant for l.c.i. morphisms of $G$-schemes. \\
$(ii) \ Homotopy :$  If $f : E \to X$ is a $G$-equivariant vector bundle,
then $f^*: \Omega^G_*(X) \xrightarrow{\cong} \Omega^G_*(E)$. \\
$(iii) \ Chern \ classes :$ For any $G$-equivariant vector bundle $E
\xrightarrow{f} X$ of rank $r$, there are equivariant Chern class operators
$c^G_l(E) : \Omega^G_*(X) \to \Omega^G_{*-l}(X)$ for $1 \le l \le r$ which
have same functoriality properties as in the non-equivariant case. \\
$(iv) \ Free \ action :$ If $G$ acts freely on $X$ with quotient $Y$, then
$\Omega^G_*(X) \xrightarrow{\cong} \Omega_*(Y)$. \\
$(v) \ Exterior \ Product :$ There is a natural product map
\[
\Omega^G_i(X) \otimes_{\Z} \Omega^G_{i'}(X') \to \Omega^G_{i+i'}(X \times X').
\]
In particular, $\Omega^G_*(k)$ is a graded algebra and $\Omega^G_*(X)$ is
a graded $\Omega^G_*(k)$-module for every $X \in \sV_G$. \\
%What is $\Omega^G_0(k)$?
$(vi) \ Projection \ formula :$ For a projective map $f : X' \to X$ in
$\sV^S_G$, one has for $x \in \Omega^G_*(X)$ and $x' \in \Omega^G_*(X')$,
the formula : $f_*\left(x' \cdot f^*(x)\right) = f_*(x') \cdot x$. \\
$(vii) \ Localization \ sequence :$ For a $G$-invariant closed subscheme
$Z \subset X$ with the complement $U$, there is an exact sequence
\[
\Omega^G_*(Z) \to \Omega^G_*(X) \to \Omega^G_*(U) \to 0.
\]
\end{thm} 

For a $G$-equivariant vector bundle $E$ on $X$, we shall often denote the
equivariant Chern class operators as $c^G_i(E) \cap -$. Note that these
Chern classes behave like the Chern classes of the ordinary vector bundles
on the ordinary cobordism of the mixed spaces defined before. 
In particular, if $\chi$ is a character of $G$ (which is just a $G$-equivariant
line bundle on ${\rm spec}(k)$), the above exterior product
is explicitly described as 
\begin{equation}\label{eqn:Module-S}
\Omega^*_G(X) \otimes_{\bL} \Omega^*(k) \to \Omega^*(X)
\end{equation}
\[
w \otimes c^G_1(L_{\chi}) \mapsto c^G_1\left(p^*(L_{\chi})\right) \cap w
\]
if $p : X \to {\rm Spec}(k)$ is the structure map. Here, $L_{\chi}$ is the 
line bundle associated to $\chi$ and $c^G_1(L_{\chi})$
is identified with the element $c^G_1(L_{\chi})(id)$ in the ring $C(G)$.

We shall denote $C(G) : = \Omega^*(k)$ by $\Omega^*(BG)$ and call it as the 
cobordism ring of the {\sl classifying space} of $G$. 
It is known from the universal
property of the algebraic cobordism that for a complex linear algebraic group
$G$, there is a natural map of rings
\begin{equation}\label{eqn:Al-Top}
\rho_G : \Omega^*(BG) \to MU^*(BG),
\end{equation}
where $MU^*(BG)$ is the rational complex cobordism group of the topological
classifying space of $G(\C)$. Moreover, this realization map is in fact 
an isomorphism ({\sl cf.} \cite[Theorem~6.8]{Krishna3}). Thus, $\Omega^*(BG)$
is truly the cobordism ring of the classifying space of $G$.

If $H \subset G$ is a closed
subgroup, then for any $G$-scheme $X$ and a good pair $(V_j, U_j)$, the 
$G/H$-fibration $X_H \to X_G$ induces a natural restriction map
\begin{equation}\label{eqn:Res}
r^G_H : \Omega^*_G(X) \to \Omega^*_H(X)
\end{equation}  
which in particular gives a natural $\Q$-algebra homomorphism
$\Omega^*(BG) \to \Omega^*(BH)$. This restriction map in fact completely 
describes $\Omega^*(BG)$ in terms of $\Omega^*(BT)$ in the following way.
\begin{thm}({\sl cf.} \cite[Theorem~6.6]{Krishna3})\label{thm:W-inv} 
The natural map $\Omega^*_G(X) \to \Omega^*_T(X)$ induces an
isomorphism of $C(G)$-modules
\[
\Omega^*_G(X) \xrightarrow{\cong} \left(\Omega^*_T(X)\right)^W.
\]
In particular, one has $C(G) \xrightarrow{\cong} {C(T)}^W$.
\end{thm}
If $T$ is a split torus of rank $r$ and if $\{\chi_1, \cdots , \chi_r\}$ is a
$\Q$-basis of $\wh{T}$ with associated line bundles $\{L_{\chi_1}, \cdots ,
L_{\chi_r}\}$, then there is a natural ring isomorphism
\[
\bL[[x_1, \cdots , x_r]] \xrightarrow{\cong} C(T)
\]
which maps $x_i$ to the class of the first Chern class $c_1(L_{\chi_i})$
in $\Omega^*(BT)$. Similarly, there is an isomorphism
$\bL[[\gamma_1, \cdots , \gamma_n]] \xrightarrow{\cong} C(GL_n)$, where
the image of $\gamma_i$ is the $i$th Chern class of the canonical rank $n$
vector bundle on $BGL_n$. Under the isomorphism $C(GL_n) \cong
C(T)^W$ of Theorem~\ref{thm:W-inv}, the image of $\gamma_i$ is the $i$th
elementary symmetric polynomial in the variables of $C(T)$.
  
We also recall here that there is a similar relation between the
Chow rings of $BG$ and $BT$ ({\sl cf.} \cite{Totaro}, \cite{EG}), that is, 
\begin{equation}\label{eqn:Res*}
S(G) = \ch^*(BG) \xrightarrow{\cong} \ch^*(BT)^W = S(T)^W
\end{equation} 
and moreover
\[
S(GL_n) \cong \Q[\gamma_1, \cdots , \gamma_n] \inj
\Q[x_1, \cdots , x_n] = S(T). 
\]

\section{The homomorphism $\lambda_X$}\label{section:prelim}
In this section, we explain the homomorphism $\lambda_X$ of
Theorem~\ref{thm:Cob-PR} and then prove some functoriality properties of 
this map with respect to the maps between schemes. We consider the case
of flag bundles associated to Borel subgroups of $G$, from which the
general case can easily be deduced ({\sl cf.} Section~\ref{section:Parabolic}).

Let $X$ be a scheme and let $p : E \to X$ be a principal $G$-bundle and let
$\pi : E/B \to X$ be the flag bundle associated to the Borel subgroup $B$.
Since $E$ is a $G$-scheme where $G$ acts freely, it follows from
Theorem~\ref{thm:Basic} that $\Omega^*(X) \cong \Omega^*_G(E)$ is a
$C(G)$-modules. In the same way, $\Omega^*(E/T) \cong \Omega^*_T(E)$ is
a $C(T)$-module (and hence a $C(G)$-module by restriction).
On the other hand, $E/T \to E/B$ is a principal $B^u$-bundle, where $B^u$ is 
the unipotent radical of $B$. By \cite[XXII, 5.9.5]{DG}, $B^u$ has a finite 
filtration by normal subgroups whose successive quotients are the vector 
groups. A successive application of homotopy invariance now implies that 
$\Omega^*(E/B) \xrightarrow{\cong} \Omega^*(E/T)$. Hence, $\Omega^*(E/B)$
is a $C(T)$-module. Thus, $\Omega^*(X)$ and $\Omega^*(E/B)$ are naturally
$C(G)$ and $C(T)$-modules respectively, which defines the map $\lambda_X$
as $\lambda_X(w \otimes x) = x \cdot \pi^*(w)$. 

Recall that $C(T)$ is the power series over $\bL$
in the first Chern classes of the line bundles associated to the characters 
of $T$ and Theorem~\ref{thm:W-inv} implies that $C(G)$ is also generated
by the first Chern classes of the $W$-invariant characters inside $C(T)$.

Using the description of these module structures in ~\eqref{eqn:Module-S},
we see that the map $\lambda_X$ is given by
\begin{equation}\label{eqn:CM1*}
\lambda_X : \Omega^*(X) \otimes_{C(G)} C(T) \to \Omega^*(E/B)
\end{equation}
\[
w \otimes c_1(\chi) \mapsto c_1(\chi) \cap {\pi}^*(w)
\]
for a character $\chi$ of $T$. 
It is easy to see from this that this is an $\bL$-algebra homomorphism
if $X$ is smooth.

\begin{remk}\label{remk:stack-way}
For readers who are little bit familiar with the language of quotient stacks
and know that the $G$-equivariant line bundles on a $G$-scheme $X$ are same as
ordinary line bundles on the quotient stack $[X/G]$, we can explain the above
in this set up as follows. The principal $G$-bundle $E \to X$ uniquely gives
rise to the following commutative diagram of morphisms.
\begin{equation}\label{eqn:stack-way1}
\xymatrix@C.6pc{
E/B \ar[r]^{\pi} \ar[d]_{q} & X \ar[d]^{\ov{p}} \\
BT \ar[r] & BG,}
\end{equation}
where $BG$ is the quotient stack $[k/G]$. The map $\lambda_X$ is then given
as $\lambda_X\left(w \otimes c_1(\chi)\right) = c_1(q^*(L_\chi)) \cap 
\pi^*(w)$.

In particular, if $E = G \times X \to X$ is a trivial principal bundle,
then the map $X \to BG$ canonically factors through the structure
map $X \to {\rm Spec}(k) \to [k/G] = BG$. Hence, the map $\lambda_X$
in this case is given by 
\begin{equation}\label{eqn:trivial}
\Omega^*(X) \otimes_{\bL} \left(\bL \otimes_{C(G)} C(T)\right) \to 
\Omega^*(E/B).
\end{equation}
Here, the left term is identified as $\Omega^*(X) \otimes_{\bL} \Omega^*(G/B)$
by \cite[Theorem~7.6]{Krishna3} and $\Omega^*(X) \otimes_{\bL} \Omega^*(G/B) 
\xrightarrow{\lambda_X} \Omega^*(E/B)$ is simply the exterior product map. 
\end{remk}
To prove certain functoriality properties of $\lambda_X$, we need the
following elementary result on the equivariant cobordism.
\begin{lem}\label{lem:simple}
Let $G$ be a linear algebraic group over $k$ and let 
$f: Y \to X$ and $g : Z \to X$ be projective and smooth morphisms of
$G$-schemes respectively. Then the maps $f_* : \Omega^*_G(Y) \to \Omega^*_G(X)$
and $g^*:\Omega^*_G(X) \to \Omega^*_G(Z)$ are $C(G)$-linear.
\end{lem}
\begin{proof}
We only give a sketch for the $C(G)$-linearity of $f_*$. 
The assertion about $g^*$ is
similar and much simpler. Note that the $C(G)$-module structure on
$\Omega^*_G(X)$ is given by the exterior product ({\sl cf.} 
Theorem~\ref{thm:Basic}). It suffices to show that 
\[
f_*(x \cdot w) = x \cdot f_*(w)
\]
when $x$ and $w$ are generators of the corresponding cobordism groups
$\Omega^*({U_j}/G)$ and $\Omega^*((Y \times U_j)/G)$, where $(V_j, U_j)$
is any given good pair for $G$-action. So let $W_1 \xrightarrow{s_1} {U_j}/G$
and $W_2 \xrightarrow{s_2} Y \stackrel{G}{\times} U_j$ be projective
morphisms from smooth and connected schemes, representing the cobordism classes
$x$ and $w$ respectively. Let $\wt{W_1}$ and $\wt{W_2}$
be the pull-backs of $W_1$ and $W_2$ to $U_j$ and $Y \times U_j$ respectively.
By the definition of the push-forward and exterior product, we have
\[
\begin{array}{lll}
f_*(x \cdot w) & = &  f_*\left(\left[\wt{W_1} \stackrel{G}{\times} \wt{W_2} 
\to Y \stackrel{G}{\times}(U_j \times U_j)\right]\right) \\
& = & \left[\wt{W_1} \stackrel{G}{\times} \wt{W_2} \to
X \stackrel{G}{\times}(U_j \times U_j)\right]
\end{array}
\] 
and the last term is same as the class of $x \cdot f_*(w)$ in $\Omega^*(X_G)$
which can be taken as $X \stackrel{G}{\times}(U_j \times U_j)$ because
$(V_j \times V_j, U_j \times U_j)$ is also a good pair for the $G$-action.
\end{proof} 
\begin{lem}\label{lem:Funct}
Let $f: Y \to X$ and $g : Z \to X$ be respectively, the projective and the 
smooth morphisms. Let $p : E \to X$ be a principal $G$-bundle and let 
$E_Y$ and $E_Z$ denote its pull-backs to $Y$ and $Z$ respectively. 
Consider the following Cartesian diagrams of flag bundles.
\begin{equation}\label{eqn:Funct1}
\xymatrix@C.6pc{
{E_Y}/B \ar[r]^{\ov{f}} \ar[d]_{\pi_Y} & E/B \ar[d]^{\pi} & & {E_Z}/B 
\ar[r]^{\ov{g}} \ar[d]_{\pi_Z} & E/B \ar[d]^{\pi} \\
Y \ar[r]_{f} & X & & Z \ar[r]_{g} & X.}
\end{equation} 
Then the diagrams
\begin{equation}\label{eqn:Funct2}
\xymatrix{
\Omega^*(Y) \otimes_{C(G)} C(T) \ar[r]^{\ \ \ \ \lambda_Y} 
\ar[d]_{f_* \otimes {\rm id}} &  
\Omega^*({E_Y}/B) \ar[d]^{{\ov{f}}_*} \\
\Omega^*(X) \otimes_{C(G)} C(T) \ar[r]_{\ \ \ \ \lambda_X} & \Omega^*(E/B)}
\end{equation}
\begin{equation}\label{eqn:Funct3}
\xymatrix{
\Omega^*(X) \otimes_{C(G)} C(T) \ar[r]^{\ \ \ \ \lambda_X} 
\ar[d]_{g^* \otimes {\rm id}} &  
\Omega^*({E}/B) \ar[d]^{{\ov{g}}^*} \\
\Omega^*(Z) \otimes_{C(G)} C(T) \ar[r]_{\ \ \ \ \lambda_Z} & \Omega^*({E_Z}/B)}
\end{equation}
are commutative.
\end{lem}
\begin{proof} To show the commutativity of the first square, we
have
\[
\begin{array}{lll}
\ov{f}_* \circ \lambda_Y(w \otimes x) & = & \ov{f}_*(x \cdot \pi^*_Y(w)) \\ 
& = & x \cdot \ov{f}_*(\pi^*_Y(w)) \hspace*{1cm} 
({\rm By \ Lemma~\ref{lem:simple}}) \\
& = & x \cdot \pi^*(f_*(w)) \\
& = & \lambda_X(f_*(w) \otimes x),
\end{array}
\]
where the third equality follows from the fact that the first square in
~\eqref{eqn:Funct1} is Cartesian with $\pi$ smooth and $f$ projective.
The proof of the commutativity of the second square is similar.
\end{proof}

Let $(G, T, B, W)$ be as above where $G$ is a connected linear algebraic 
group. Let $G^u$ denote the unipotent radical of $G$ and let $L$ denote
the corresponding quotient as a reductive group. Then any principal $G$-bundle
$E \to X$ canonically gives a principal $L$-bundle $E_L = E/{G^u} \to X$.
Moreover, as the Borel subgroup $B$ contains $G^u$, we see that
$E/B \xrightarrow{\cong} {E_L}/{B_L}$, where $B_L$ is the image of $B$ which
is a Borel subgroup of the reductive group $L$. Since $C(G) \cong C(L)$,
as follows from the Levi decomposition and the homotopy invariance, we
conclude that it is enough to consider the case when $G$ is reductive
in order to prove our main results. Hence for the rest of this paper, $G$
will always denote a connected reductive group. 
We shall deduce Theorem~\ref{thm:Cob-PR} from the following result
for the algebraic cobordism of the flag bundles associated to the Borel
subgroup $B$.
\begin{thm}\label{thm:Cob-Main}
Let $p : E \to X$ be a principal $G$-bundle and let $\pi :E/B \to X$ be
the flag bundle associated to the Borel subgroup $B$.
The natural map of $C(T)$-modules
\begin{equation}\label{eqn:CM1}
\lambda_X : \Omega^*(X) \otimes_{C(G)} C(T) \to \Omega^*(E/B)
\end{equation}
is an isomorphism. Moreover, it is an isomorphism of rings if $X$ is smooth.
\end{thm}

\section{Some algebraic reductions}\label{section:ALG-R}
Let $(G, T, B, W)$ be as above and let $T$ be a split torus of rank $r$. This
rank will be fixed throughout. We fix a basis $\{\chi_1, \cdots , \chi_r\}$
of $\wh{T}$ and let $S = \Sym(\wh{T}) = \Q[x_1, \cdots , x_r]$ be the 
polynomial algebra in the first Chern classes of the line bundles associated 
to the characters $\{\chi_i, \cdots , \chi_r\}$. Let $S^W \subset S$ be
the subalgebra generated by the homogeneous polynomials which are invariant
under the action of $W$. This gives us a square of ring inclusions
\begin{equation}\label{eqn:PPS}
\xymatrix@C.7pc{
(\bL[x_1, \cdots , x_r])^W \ar[r] \ar[d] & \bL[x_1, \cdots , x_r] \ar[d] \\
C(G) \ar[r] & C(T),}  
\end{equation}
which is Cartesian and where $C(G)$ has been identified with $C(T)^W$. 
We shall write $\bL[x_1, \cdots , x_r]$ simply as $S_{\bL}$.
Note that $S_{\bL}$ and $S_{\bL}^W$ are canonically isomorphic to
$\bL \otimes_{\Q} S$ and $\bL \otimes_{\Q} S^W$ as $\bL$-algebras.
It is also known that $S^W_{\bL}$ is a polynomial algebra over $\bL$
of rank $r$. We shall denote the homogeneous generators of this subalgebra
by $\{\sigma_1, \cdots , \sigma_r\}$.

Let $I$ be the ideal of $S$ generated by the homogeneous elements of
positive degree which are invariant under $W$ and let $\Lambda$ denote the
ring $S/I$. Then we see that $\Lambda_{\bL} = \bL \otimes_{\Q} \Lambda$
is canonically isomorphic to the $\bL$-algebra $\bL[x_1, \cdots , x_r]/I$.
We recall the following result from \cite[Lemma~1.2]{Vistoli}.
\begin{lem}\label{lem:VIS}
The graded $\Q$-algebra $\Lambda$ is finite. If $N$ is the maximal integer
for which $\Lambda_N \neq 0$, then $N = {\rm dim}(G/B)$. Moreover, the
$\Q$-vector space $\Lambda_N$ is one-dimensional, and if $d$ is an integer,
the homomorphism 
\[
\Lambda_d \otimes \Lambda_{N-d} \to \Lambda_N
\]
given by the multiplication in $\Lambda$ is a perfect pairing of 
finite-dimensional $\Q$-vector spaces.
\end{lem}
\begin{lem}\label{lem:elem-comp}
Let $A$ be a commutative ring and let $I$ be an ideal of $A$. Let $J$ be a 
finitely generated ideal of $A$. Then for any $A$-module $M$, the natural
maps of $\wh{A}$-modules
\[
J\wh{M} \to \wh{JM}, \ \ \frac{\wh{M}}{J\wh{M}} \to
\wh{\left(\frac{M}{JM}\right)}
\]
are isomorphisms, where $\wh{M}$ denotes the $I$-adic completion of $M$.
\end{lem}
\begin{proof} Consider the exact sequence 
\begin{equation}\label{eqn:elem-comp1}
0 \to JM \to M \to \frac{M}{JM} \to 0.
\end{equation}
Since the topology on $M$ is given by the descending chain 
$ M \supset IM \supset I^2M \supset \cdots$ of submodules, it follows from
\cite[Theorem~8.1]{Matsumura} that 
\begin{equation}\label{eqn:elem-comp2}
0 \to \wh{(JM)} \to \wh{M} \to \wh{\left(\frac{M}{JM}\right)} \to 0
\end{equation}  
is exact. Thus, we only need to show the first isomorphism to prove the
lemma.

Suppose $J = \stackrel{n}{\underset{i =1}\sum}a_iA$ and define
\[
\phi : M^n \to M
\]
\[
\phi(m_1, \cdots , m_n) =  \stackrel{n}{\underset{i =1}\sum} a_im_i.
\]
This makes the sequence 
\begin{equation}\label{eqn:elem-comp3}
M^r \to M \to \frac{M}{JM} \to 0
\end{equation}
exact. It again follows from \cite[Theorem~8.1]{Matsumura} that
\begin{equation}\label{eqn:elem-comp4}
{\wh{M}}^r \xrightarrow{\wh{\phi}} \wh{M} \to 
\wh{\left(\frac{M}{JM}\right)} \to 0
\end{equation}
is exact. On the other hand, $\wh{\phi}$ is again given by 
$\wh{\phi}(\wh{m_1}, \cdots , \wh{m_n}) =
 \stackrel{n}{\underset{i =1}\sum} a_i\wh{m_i}$. In other words,
${\rm Image}(\wh{\phi}) = J\wh{M}$. The first isomorphism now follows from this
and ~\eqref{eqn:elem-comp2}. This proves the lemma.
\end{proof}  
\begin{cor}\label{cor:Finite}
The natural homomorphisms of rings
\[
\bL \otimes \Lambda \xrightarrow{\cong}
\frac{\bL[x_1, \cdots , x_r]}{I} \to \frac{C(T)}{IC(T)} \to C(T) \otimes_{C(G)}
\bL
\]
are isomorphisms.
\end{cor} 
\begin{proof} We first observe that the cobordism ring $C(T)$ is the inverse 
limit of the cobordism rings of the form ${\left(\Omega^*(BT)\right)}_{j\ge 0}$
on each of which the Weyl group acts. In particular, the action of $W$ on
$C(T)$ is induced by its action on the polynomial ring $\bL[x_1, \cdots, x_r]$
and $C(T)^W$ is the inverse limit of the $W$-invariants in the inverse system
${\left(\Omega^*(BT)\right)}_{j\ge 0}$. Thus we see that we can write
$S^W_{\bL} = \bL[\sigma_1, \cdots , \sigma_r] \inj \bL[x_1, \cdots , x_r]$
and $C(G) = {C(T)}^W$ is the subring of the power series ring 
$\bL[[x_1, \cdots , x_r]]$ generated by the homogeneous polynomials
$\{\sigma_1, \cdots , \sigma_r\}$. Moreover, the ideal $I$ in $C(T)$ 
is the extension of the ideal $(\sigma_1, \cdots , \sigma_r)$ of $S^W
= \Q[\sigma_1, \cdots ,\sigma_r]$ which we also denote by $I$. 
Let $\mathfrak{m}$ denote the ideal $(x_1, \cdots , x_r)$ of $S_{\bL}$.
Now we have
\[
\begin{array}{lll}
C(T) \otimes_{C(G)} \bL & \cong & C(T) \otimes_{C(G)}
\left(\frac{C(G)}{(\sigma_1, \cdots , \sigma_r)}\right) \\
& \cong & \frac{C(T)}{IC(T)} \\
& \cong & \frac{{\wh{({S_{\bL}})}}_{\mathfrak{m}}}
{I{\wh{({S_{\bL}})}}_{\mathfrak{m}}} \\ 
& \cong & {\wh{\left(\frac{S_{\bL}}{IS_{\bL}} \right)}}_{\mathfrak{m}}
\hspace*{2cm} ({\rm By \ Lemma~\ref{lem:elem-comp}}).
\end{array}
\]
On the other hand, 
\begin{equation}\label{eqn:complete}
\frac{S_{\bL}}{IS_{\bL}} \cong \bL \otimes_{\Q} 
\left(\frac{\Q[x_1, \cdots , x_r]}{I}\right) \cong 
\stackrel{s}{\underset{j = 1}\prod} \bL \otimes_{\Q} A_j,
\end{equation}
where each $A_j$ is an artinian local ring which is finite over
$\Q$. In particular, the ideal $\mathfrak{m}$ is nilpotent in each of the
factor $\bL \otimes A_j$ and hence the last term in ~\eqref{eqn:complete} 
is complete with respect to $\mathfrak{m}$. 
We conclude that $\frac{S_{\bL}}{IS_{\bL}}$ is complete in the 
$\mathfrak{m}$-adic topology. In particular, we obtain
\[
C(T) \otimes_{C(G)} \bL \cong 
{\wh{\left(\frac{S_{\bL}}{IS_{\bL}} \right)}}_{\mathfrak{m}}
\cong \frac{S_{\bL}}{IS_{\bL}}
\]
and this completes the proof.
\end{proof}
\begin{cor}\label{cor:finite-trivial}
Let $X \times G/B \xrightarrow{\pi} X$ be the trivial flag bundle.
Then the map $\lambda_X$ is given by 
\[
\Omega^*(X) \otimes_{\Q} \Lambda \xrightarrow{\cong} \Omega^*(X) 
\otimes_{\bL} \Lambda_{\bL} \xrightarrow{\lambda_X}
\Omega^*(E/B)
\]
which is an $\bL$-algebra isomorphism if $X$ is smooth.
\end{cor}
\begin{proof}
The first assertion of the corollary follows directly from 
~\eqref{eqn:trivial} and Corollary~\ref{cor:Finite}. If $X$ is smooth, this
map is an $\bL$-algebra homomorphism because so are the maps in 
~\eqref{eqn:CM1*} and Corollary~\ref{cor:Finite}. Moreover, it is
an isomorphism by  \cite[Lemma~6.5]{Krishna3} and 
\cite[Theorem~3.1]{Levine1}.
\end{proof}

\section{Surjectivity of $\lambda_X$}\label{section:Surj}
We now let $p : E \to X$ be a an arbitrary principal $G$-bundle and let 
$\pi: E/B \to X$ be the associated flag bundle.
Using the above inclusions of the polynomial rings inside the 
power series rings, we get natural homomorphisms 
\begin{equation}\label{eqn:REDP}
\xymatrix{
\Omega^*(X) \otimes_{S^W_{\bL}} S_{\bL} \ar[dr]^{\ \ \ \ \phi_X} 
\ar[d]_{\delta_X} & \\
\Omega^*(X) \otimes_{C(G)} C(T) \ar[r]_{\ \ \ \ \ \lambda_X} & \Omega^*(E/B)}
\end{equation} 
of $S_{\bL}$-modules, which are also $S_{\bL}$-algebra homomorphisms if
$X$ is smooth.

As a first step towards proving Theorem~\ref{thm:Cob-Main}, we show in this
section that the
map $\lambda_X$ is surjective. In fact, the proof that follows will 
show that the map $\phi_X$ is surjective. It will eventually turn out that 
both the maps $\phi_X$ and $\lambda_X$ are isomorphisms. We begin with
the following elementary property of principal bundles and the
local property of algebraic cobordism.
\begin{lem}\label{lem:trivial-flag}
Let $p : E \to X$ be a principal $G$-bundle and let $\pi : E/B \to X$ be the
associated flag bundle. Then the $G$-action $G \times E \xrightarrow{\mu} E$
induces a commutative diagram
\begin{equation}\label{eqn:tr-flag1}
\xymatrix{
G/B \times E/B \ar[r]^{\ \ \ \ {\mu}'} \ar[d]_{{\pi}'} & E/B \ar[d]^{\pi} \\
E/B \ar[r]_{\pi} & X}
\end{equation}
which is Cartesian and where ${\pi}'$ is the projection to the second factor.
\end{lem}
\begin{proof}
Since $E \xrightarrow{p} X$ is a principal $G$-bundle quotient of 
quasi-projective schemes, the action map $G \times E \xrightarrow{\mu} E$
induces a commutative diagram
\begin{equation}\label{eqn:tr-flag2}
\xymatrix{
G \times E \ar[r]^{\ \ \ \ \mu} \ar[d]_{p'} & E \ar[d]^{p} \\
E \ar[r]_{p} & X}
\end{equation}
which is Cartesian and where $p'$ is the projection to the second factor 
by the general properties of principal bundles ({\sl cf.} \cite[0.10]{GIT}).

Now, the map $\mu$ descends to a map $G \times E/B \to E/B$. Moreover,
this map is $B$-equivariant where $B$ acts trivially on $E/B$ and by left
multiplication on $G$. Taking the quotients, we get a canonical
map $G/B \times E/B \xrightarrow{{\mu}'} E/B$ making the diagram 
~\eqref{eqn:tr-flag1} commute. It is now an easy exercise to check from 
~\eqref{eqn:tr-flag2} that this diagram is Cartesian too. 
\end{proof}

\begin{lem}\label{lem:star} 
Let $f : X' \to X$ be a finite and \'etale morphism of smooth and connected
schemes. Then there exists an open subscheme $U \stackrel{j}{\inj} X$
such that for the map $g = f|_{U'} : U' = f^{-1}(U) \to U$, one has
$g_*(1) = [k(X') : k(X)]$.
\end{lem}
\begin{proof}
Let $\eta$ denote the generic point of $X$ and consider the Cartesian diagram
\begin{equation}\label{eqn:star1}
\xymatrix@C.6pc{
X'_{\eta} \ar[d]_{h} \ar[r]^{i'} & U' \ar[d]^{g} \ar[r]^{j'} & X' \ar[d]^{f} \\
\eta \ar[r]_{i} & U \ar[r]_{j} & X,}
\end{equation}
where $U$ is any open subscheme of $X$. Since $X'$ is connected, we see
that $X_{\eta} = {\rm Spec}(k(X'))$. Put $p = j \circ i$,
$p' = j' \circ i'$ and $d = [k(X') : k(X)]$. It follows from 
\cite[Lemma~4.7]{LM1} that 
\[
p^* \circ f_* (1) = h_* \circ p'^*(1) = h_*(1) = d.
\]
Since the algebraic cobordism is generically constant by 
\cite[Lemma~13.3, Corollary~13.4]{LM1}, there exists an open subscheme 
$U \stackrel{j}{\inj} X$ such that $j^* \circ f_*(1) = d$ in $\Omega^*(U)$.
This in turn implies that 
\[
g_*(1) = g_* \circ j'^*(1) = j^* \circ f_*(1) = d
\]
and this proves the lemma.
\end{proof}
\begin{cor}\label{cor:Retract}
Let $f : X' \to X$ be a finite and \'etale morphism of smooth and connected
schemes and consider the diagram ~\eqref{eqn:Funct1}. Then there exists an
open subscheme $U \stackrel{j}{\inj} X$ such that for 
$g = f|_{U'} : U' = f^{-1}(U) \to U$, one has a commutative diagram
\begin{equation}\label{eqn:Retract1}
\xymatrix@C.6pc{
\Omega^*(U) \otimes_{C(G)} C(T) \ar[d]_{\ \ \ \ \lambda_U}
\ar[r]^{g^*} & \Omega^*(U') \otimes_{C(G)} C(T) \ar[d]_{\ \ \ \ \lambda_{U'}} 
\ar[r]^{g_*} & \Omega^*(U) \otimes_{C(G)} C(T) \ar[d]_{\ \ \ \ \lambda_U} \\
\Omega^*({E_U}/B) \ar[r]_{\ov{g}^*} & \Omega^*({E_{U'}}/B) \ar[r]_{{\ov{g}}_*}
& \Omega^*({E_U}/B)}
\end{equation}
such that the horizontal composite maps are multiplication
by $[k(X') : k(X)]$.
\end{cor}
\begin{proof}
We choose  $U \stackrel{j}{\inj} X$ as in Lemma~\ref{lem:star}.
The commutativity of the diagram follows from Lemma~\ref{lem:Funct}.
Moreover, as $f$ is finite and \'etale of degree $d$, it follows that
$\ov{f}$ is also a morphism of the same type. 
We claim that $\ov{g}_*(1) = k(X') : k(X)]$. To see this, we 
evaluate the required term as
\[
\ov{g}_*(1) = \ov{g}_* \circ {\pi}^*_{U'} (1) = {\pi}^*_U \circ g_*(1)
= [k(X') : k(X)],
\]
where the second equality follows from Lemma~\ref{lem:star}
and this proves the claim.  The corollary now follows from the
projection formula.
\end{proof} 
% the projection formula can be locally proven in the same way as
% Lemma~\ref{lem:star}. That is, we can prove that on some open subset,
% it is $f_* \circ f^*$ is multiplication by degree.
% It is very easy to check for finite map that $f_* \circ f^* (x) = f_*(1)
% \cdot x$.

\begin{prop}\label{prop:Surj}
The map $\lambda_X$ is surjective for any scheme $X$.
\end{prop}
\begin{proof} We shall prove this by induction on the dimension of $X$.
We can assume that $X$ is reduced. If $X$ is zero-dimensional, it is of the
form $X = {\rm Spec}(K)$, where $K$ is a finite product of
finite field extensions of $k$. We prove the case when $X = {\rm Spec}(k)$. 
The same proof applies for any finite extension of $k$.
Now, there is a finite extension $k \inj l$ such that $Y_{l}$ is of the form
$G_{l}/{B_{l}}$. Hence the result holds for $X = {\rm Spec}(l)$ by  
\cite[Theorem~7.6]{Krishna3}. The case of ${\rm Spec}(k)$ now follows from
Corollary~\ref{cor:Retract}.

If the map $\pi$ is of the form $X \times G/B \xrightarrow{\pi} X$ with $X$
smooth, the maps $\lambda_X$ and $\phi_X$ are in fact isomorphisms by
~\eqref{eqn:trivial}, Corollary~\ref{cor:finite-trivial},
\cite[Lemma~6.5]{Krishna3} and \cite[Theorem~3.1]{Levine1}. In the general
case, we can find an \'etale cover $X' \xrightarrow{f} X$ such that
the base change $E/B {\times}_X X' \xrightarrow{{\pi}'} X'$ is the
trivial flag bundle $X' \times G/B \to X'$.
We can now find a smooth and dense open subset $U \stackrel{j}{\inj} X$
such that the map $U' = f^{-1}(U) \to U$ is finite and \'etale. Moreover,
the flag bundle is still trivial on $U'$. Since $U$ is a disjoint union of
smooth and connected schemes, the surjectivity of $\lambda_U$ follows from
the case of the trivial bundle shown above and Corollary~\ref{cor:Retract}.

We now let $Z = X - U$ be the complement of $U$ in $X$ with the reduced
closed subscheme structure and consider the diagram
\begin{equation}\label{eqn:induction}
\xymatrix@C.5pc{
\Omega^*(Z) \otimes_{C(G)} C(T) \ar[r] \ar[d]_{\lambda_Z} &
\Omega^*(X) \otimes_{C(G)} C(T) \ar[r] \ar[d]_{\lambda_X} &
\Omega^*(U) \otimes_{C(G)} C(T) \ar[r] \ar[d]^{\lambda_U} & 0 \\
\Omega^*({E_Z}/B) \ar[r] & \Omega^*(E/B) \ar[r] & 
\Omega^*({E_U}/B) \ar[r] & 0}
\end{equation}
which is commutative by Lemma~\ref{lem:Funct} and whose rows are exact by 
Theorem~\ref{thm:Levine-M}.
Since $U$ is open and dense, $Z$ is a closed
subscheme of dimension which is strictly smaller than that of $X$. Hence
the map $\lambda_Z$ is surjective by induction. We have shown above that
$\lambda_U$ is surjective. Hence the map $\lambda_X$ is surjective too.
\end{proof}
\begin{remk}\label{remk:Surj-General}
Since the maps 
\[
\Omega^*(X) \otimes_{S^W_{\bL}} S_{\bL} \xrightarrow{\delta_X} 
\Omega^*(X) \otimes_{C(G)} C(T) \xrightarrow{\lambda_X}  \Omega^*(E/B)
\]
are isomorphisms for the trivial bundle $G/B \times X \xrightarrow{\pi} X$
for $X$ smooth by Corollary~\ref{cor:finite-trivial}, 
exactly the same proof as for Proposition~\ref{prop:Surj}
shows that the maps $\delta_X$ and $\phi_X$ are also surjective for any
scheme $X$.
\end{remk}

\section{Proof of Theorem~\ref{thm:Cob-Main}}\label{section:ProofI}
Recall from Section~\ref{section:ALG-R} that $\Lambda = S/I
= \Lambda_0 \oplus \Lambda_1 \oplus \cdots  \oplus \Lambda_N$ is the graded
quotient of $S = \Q[x_1, \cdots , x_r]$ by the ideal $I$ which is 
generated by the homogeneous polynomials of positive degree which are 
invariant under $W$. It is clear that
$\Lambda_0$ is one-dimensional over $\Q$ generated by the unit element
of the ring. It also follows from Lemma~\ref{lem:VIS} that
$\Lambda_N$ is an one-dimensional $\Q$-vector space. 
We fix these two generators and denote them by $\rho_0 = 1$ and 
$\rho_N$ respectively. Let $p_0 = 1$ and $p_N$ be their homogeneous
lifts in $S_0$ and $S_N$ respectively. For any scheme $X$, we consider the
map
\begin{equation}\label{eqn:PSI}
\psi : \Omega^*(X) \to \Omega^*(X)
\end{equation}
\[
\psi(x) = \pi_* \circ \phi_X(x \otimes p_N) = \pi_*\circ\left(c_1(p_N) \cap
\pi^*(x)\right)
\]
where $\phi_X$ is the homomorphism in ~\eqref{eqn:REDP}.
We need the following property of this map.
\begin{lem}\label{lem:Closed}
Let $Z \stackrel{i}{\inj} X$ be a closed subscheme such that $\psi_Z$ is
identity. Let $\Omega^*_Z(X) \subset \Omega^*(X)$ be the image of the 
map $\Omega^*(Z) \xrightarrow{i_*} \Omega^*(X)$. Then, $\psi_Z$ induces
a map $\psi^Z_X : \Omega^*_Z(X) \to \Omega^*_Z(X)$ such that the diagram
\begin{equation}\label{eqn:Closed1}
\xymatrix{
\Omega^*_Z(X) \ar[r] \ar[d]_{\psi^Z_X} & \Omega^*(X) \ar[d]^{\psi_X} \\
\Omega^*_Z(X) \ar[r] & \Omega^*(X)}
\end{equation}
is commutative and $\psi^Z_X$ is identity.
\end{lem}
\begin{proof}
We consider the following diagram.
\[
\xymatrix@C.9pc{
\Omega^*(Z) \ar[d]_{i_*} \ar[r]^{\pi^*_Z} & \Omega^*({E_Z}/B) \ar[d]^{i'_*}
\ar[r]^{\cdot p_N} & \Omega^*({E_Z}/B) \ar[d]^{i'_*} 
\ar[r]^{\ \ \ {\pi}_{Z_*}} & \Omega^*(Z) \ar[d]^{i_*} \\
\Omega^*(X) \ar[r]_{\pi^*} & \Omega^*({E}/B) 
\ar[r]_{\cdot p_N} & \Omega^*({E}/B)
\ar[r]_{\ \ \pi_*} & \Omega^*(X).}
\]   
The first square from the left clearly commutes as $\pi_Z$ is the pull-back of
a smooth morphism. The second square commutes by Lemma~\ref{lem:simple}.
The third square is simply the commutativity of the push-forward maps.
We conclude that the big outer square commutes.
Since the top and the bottom composite horizontal arrows are $\psi_Z$
and $\psi_X$ respectively, we get the commutative square ~\eqref{eqn:Closed1}.
Moreover, $\psi^Z_X$ is identity because $\psi_Z$ is so.
\end{proof}
 
%The second square also commutes by the projection formula
%for the action of first Chern classes of line bundles, 

\begin{lem}\label{lem:Main-Id}
For any scheme $X$, there is a finite filtration by closed subschemes
\[
\emptyset = X^{n+1} \subset X^n \subset \cdots \subset X^1 \subset X^0 = X
\]
such that $\psi_{U^i}$ is identity for each $0 \le i \le n$, where
$U^i = \left(X^i - X^{i+1}\right)$.
\end{lem}
\begin{proof}
We prove this by induction on the dimension of $X$.  
If $X$ is zero-dimensional, we can use Corollary~\ref{cor:Retract} and the 
argument in the proof of Proposition~\ref{prop:Surj} to reduce to the case when
$X = {\rm Spec}(k)$ and $E/B = G/B$. We can apply 
Corollary~\ref{cor:Finite} and \cite[Theorem~7.6]{Krishna3} to get a graded
$\bL$-algebra isomorphism
\begin{equation}\label{eqn:M-one}
\phi_k: \bL \otimes_{\Q} \Lambda \xrightarrow{\cong} \Omega^*(G/B).
\end{equation}
Let $S_{\bL} \xrightarrow{\eta} \Lambda_{\bL}$ denote the quotient map.  
Under the above isomorphism, we get for any homogeneous element $m \in \bL^i$,
\[
\psi_k(m) = \pi_* \circ \phi_k \circ \eta
\left(m \otimes p_N\right) =
\pi_* \circ \phi_k \left(m \otimes \rho_N\right) = 
\pi_*\left(\rho_N \cdot \pi^*(m)\right) = 
\pi_*\left(1 \otimes \rho_N\right) \cdot m,
\]
where the last equality holds by the projection formula.
On the other hand, as $\bL^0$ is the one-dimensional vector space 
generated by the class of $[{\rm Spec}(k)]$ , we see that
$\Omega^N(G/B) \cong \bL^0 \otimes \Lambda_N 
\xrightarrow{\cong} \Q[1 \otimes \rho_N]$ and 
$\pi_*(1 \otimes p_N)$ is simply the class of $[{\rm Spec}(k)]$ in $\bL^0$,
which is the identity element. This proves the zero-dimensional case.
If $X$ is any smooth scheme and $Y =
X \times G/B$, then it follows from Corollary~\ref{cor:finite-trivial} and 
the case of $X = {\rm Spec}(k)$ that $\psi_X$ is identity.

In the general case, we can find an \'etale cover $f : X' \to X$ such that
$E$ is trivial over $X$. Hence $E/B \times_X X' \xrightarrow{\cong}
E/B \times X'$ and $\pi_{X'}$ is just the projection map by 
Lemma~\ref{lem:trivial-flag}. Since $f$ is generically finite, we can find
a dense open subset $U \stackrel{j}{\inj} X$ such that $U$ is smooth and
the map $g: U' = f^{-1}(U) \to U$ is finite and \'etale. The proof of
Lemma~\ref{lem:Closed} shows that the right square in the diagram
\begin{equation}\label{eqn:Main-Id1}
\xymatrix@C.7pc{
\Omega^*(U) \ar[r]^{g^*} \ar[d]_{\psi_U} & \Omega^*(U') \ar[r]^{g_*}
\ar[d]^{\psi_{U'}} & \Omega^*(U) \ar[d]^{\psi_U} \\
\Omega^*(U) \ar[r]_{g^*} & \Omega^*(U') \ar[r]_{g_*}
& \Omega^*(U)}
\end{equation}
commutes. Since $g$ is \'etale, the similar argument shows that the left
square commutes. Furthermore, we can apply Corollary~\ref{lem:star} to
choose the open subset $U$ so that $g^*$ is injective.
We have shown above that $\psi_{U'}$ is identity.
We conclude that $\psi_U$ must also be identity.

Put $X^1 = (X-U)$. Then $X^1$ is a closed subscheme of $X$ of dimension
which is strictly less than that of $X$. The proof of the lemma now follows
by induction.
\end{proof}  
\begin{prop}\label{prop:MAIN*}
For any scheme $X$, the homomorphism $\psi_X$ is an isomorphism.
\end{prop}
\begin{proof}
We choose a finite filtration of $X$ as in Lemma~\ref{lem:Main-Id}.
We have shown that $\psi_{(X - X^1)}$ is identity. Assume by induction that
$\psi_{(X - X^m)}$ is an isomorphism and consider the diagram
\begin{equation}\label{eqn:Main*1}
\xymatrix@C.5pc{
0 \ar[r] & \Omega^*_{(X^m - X^{m+1})}\left(X - X^{m+1}\right) \ar[r] 
\ar[d]_{\psi^{(X^m - X^{m+1})}_{(X - X^{m+1})}} &
\Omega^*(X - X^{m+1}) \ar[r] \ar[d]^{\psi_{(X - X^{m+1})}} &
\Omega^*(X - X^m) \ar[d]^{\psi_{(X - X^m)}} \ar[r] & 0 \\
0 \ar[r] & \Omega^*_{(X^m - X^{m+1})}\left(X - X^{m+1}\right) \ar[r] &
\Omega^*(X - X^{m+1}) \ar[r] & \Omega^*(X - X^m) \ar[r] & 0}
\end{equation}
which is commutative by Lemma~\ref{lem:Closed}.
The top and the bottom rows are exact by Theorem~\ref{thm:Levine-M}.
The left vertical map is identity by Lemmas~~\ref{lem:Closed} and
\ref{lem:Main-Id}. The right
vertical map is isomorphism by induction. We conclude that the middle
vertical map is an isomorphism too and the induction continues.
\end{proof}
\begin{cor}\label{cor:MAIN-INJ}
For any scheme $X$, the pull-back map $\Omega^*(X) \xrightarrow{\pi^*}
\Omega^*(E/B)$ is injective.
\end{cor}
\begin{proof}
This follows directly from Proposition~\ref{prop:MAIN*}.
\end{proof}
\begin{prop}\label{prop:Trivial-ISO}
Let $X \times G/B \xrightarrow{\pi} X$ be the trivial flag bundle.
Then the map $\lambda_X$ is an isomorphism.
\end{prop}
\begin{proof}
The surjectivity of $\lambda_X$ and $\delta_X$ follows directly from
Proposition~\ref{prop:Surj} and Remark~\ref{remk:Surj-General}.
So we only need to show the injectivity.
By Remark~\ref{remk:Surj-General} again, it suffices to show that $\phi_X$ is
injective (and hence isomorphism). 
By Corollary~\ref{cor:finite-trivial}, $\phi_X$ is same
as the map of graded $\Lambda_{\bL}$-modules
\begin{equation}\label{eqn:TR-ISO1}
\Omega^*(X) \otimes_{\Q} \Lambda \xrightarrow{\phi_X} \Omega^*(E/B) 
\end{equation}
and this map factorizes the composite map $\theta_X$ as
\begin{equation}\label{eqn:TR-ISO2}
\theta_X : \Omega^*(X) \otimes_{\bL} S_{\bL} \stackrel{\eta_X}{\surj}
\Omega^*(X) \otimes_{\bL} \Lambda_{\bL} \xrightarrow{\phi_X} \Omega^*(E/B).
\end{equation}
 
It suffices to show that $\phi_X$ is injective on each graded piece of the 
left term in ~\eqref{eqn:TR-ISO1}.
So let 
\[
x \in {\left(\Omega^*(X) \otimes_{\Q} \Lambda\right)}^m =
\stackrel{N}{\underset{p = 0}\oplus} \Omega^{m-p} \otimes_{\Q} \Lambda_p.
\] be such that $\phi_X(x) = 0$.

Let $\{b^p_1, \cdots , b^p_{s_p}\}$ be a chosen $\Q$-basis of $\Lambda_p$ and 
let $\{u^p_1, \cdots , u^p_{s_p}\}$ their homogeneous lifts in $S$. 
We have $b^0_1 = \rho_0 =1$ and $b^N_1 = \rho_N$. Similarly, $u^0_1 = p_0 = 1$
and $u^N_1 = p_N$. We can then write $x$ uniquely as
\[
x = \stackrel{N}{\underset{p = 0}\sum} x_p, \ \ {\rm where} \ \
x_p = \stackrel{s_p}{\underset{i = 1}\sum} x^p_i \otimes b^p_i.
\]
We show inductively that $x_p = 0$ for each $p$. First of all,
\[
\phi_X(x) = 0 \Rightarrow p_N \cdot \phi_X(x) = 0 \Rightarrow 
\theta_X(p_N \cdot x) = 0 \Rightarrow \phi_X (\rho_N \cdot x) = 0.
\]
On the other hand, we have by Lemma~\ref{lem:VIS},
\[
\rho_N \cdot x_p = \stackrel{s_p}{\underset{i = 1}\sum} 
x^p_i \otimes (b^p_i \cdot \rho_N)
\]
which is zero for $p > 0$ and $x^0_1 \otimes \rho_N$ for $p =0$.
Thus we conclude that 
\[
\phi_X(x) = 0 \Rightarrow \phi_X(x^0_1 \otimes \rho_N) = 0 \Rightarrow
\pi_* \circ \phi_X(x^0_1 \otimes \rho_N) = 0.
\]
But the last equality is equivalent to saying that $\psi_X(x^0_1) = 0$ and
this implies that $x^0_1 = 0$ by Proposition~\ref{prop:MAIN*}.
In particular, $x_0 = x^0_1 \otimes b^0_1 = 0$.

We assume by induction that $x_q = 0$ for $q < p$ with $p \ge 1$ and
we show that $x_p$ is zero too. We prove this following the proof of
\cite{Vistoli} for the Chow groups. We fix an integer $1 \le l \le s_p$.
By Lemma~\ref{lem:VIS}, there exists $c \in \Lambda_{N-p}$ such that
\begin{eqnarray*}
b^p_i \cdot c = & \tuborg  0 \ {\rm if} \ i \neq l \\
\rho_N \ {\rm if} \ i = l.
\sluttuborg
\end{eqnarray*}
In particular, we get 
\[
\begin{array}{lll}
c \cdot x & = & \stackrel{N}{\underset{j= p}\sum} \
\stackrel{s_j}{\underset{i= 1}\sum} x^j_i \otimes (b^j_i \cdot c) \\
& = & \stackrel{s_p}{\underset{i= 1}\sum} x^p_i \otimes (b^p_i \cdot c) \\
& = & x^p_l \otimes \rho_N, 
\end{array}
\]
where the second equality occurs because $j+ (N-p) > N$ for $j > p$.
We lift $c$ to a homogeneous element $e$ of $S$. We then have as before,
\begin{equation}\label{eqn:VIS*1}
\phi_X(x) = 0 \Rightarrow e \phi_X(x) = 0 \Rightarrow
\theta_X(e \cdot x) = 0 \Rightarrow \phi_X (c \cdot x) = 0.
\end{equation}
We conclude from this that 
$\phi_X(x) = 0 \Rightarrow \phi_X(x^p_l \otimes \rho_N) = 0$, which in turn
implies that
\[
\psi_X(x^p_l) = \pi_* \circ \phi_X \left(c_1(p_N) \cap \pi^*(x^p_l)\right)
= \pi_* \circ \phi_X(x^p_l \otimes \rho_N) = 0.
\]
It follows from Proposition~\ref{prop:MAIN*} that $x^p_l = 0$ and the
induction continues to show that $x_p = 0$. This completes the proof.
\end{proof}
\noindent
{\bf {Proof of Theorem~\ref{thm:Cob-Main}:}} 
The surjectivity of $\lambda_X$ follows from Proposition~\ref{prop:Surj}.
Furthermore, the map $\delta_X$ is also surjective by 
Remark~\ref{remk:Surj-General}. Thus we only need to show that $\phi_X$ is
injective.

Put $Y = E/B$. It follows from Lemma~\ref{lem:trivial-flag} that the pull-back
$Y \times_X Y \to Y$ is the trivial flag bundle $G/B \times Y 
\xrightarrow{{\pi}'} Y$. We now consider the diagram
\begin{equation}\label{eqn:CBM1}
\xymatrix{
\Omega^*(X) \otimes_{S^W_{\bL}} S_{\bL} \ar[r]^{\ \ \phi_X} 
\ar[d]_{\pi^* \otimes 
{\rm Id}} & \Omega^*(E/B) \ar[d]^{{\pi}'^*} \\
\Omega^*(Y) \otimes_{S^W_{\bL}} S_{\bL} \ar[r]_{\phi_Y} & \Omega^*(G/B \times
Y)}
\end{equation}
which is commutative by Lemma~\ref{lem:Funct}.
Since $S_{\bL}$ is flat over $S^W_{\bL}$, the left vertical arrow is injective
by Corollary~\ref{cor:MAIN-INJ}. The bottom horizontal arrow is injective by
Proposition~\ref{prop:Trivial-ISO}. We conclude that the top horizontal
arrow is injective too.
$\hfil \square$
\begin{cor}\label{cor:Cob-Group}
Let $G$ be a connected linear algebraic group with a split maximal torus $T$.
If $E \xrightarrow{p} X$ is a principal $G$-bundle over a scheme $X$, then
\[
\Omega^*(E) \xrightarrow{\cong}
\Omega^*(X) \otimes_{C(G)} \bL \cong \frac{\Omega^*(X)}{I\Omega^*(X)},
\]
where $I$ is the ideal of $C(G)$ generated by the Chern classes of
$G$-homogeneous line bundles. This is an $\bL$-algebra isomorphism if $X$ is
smooth.
\end{cor}
\begin{proof}
By \cite[Theorem~7.4]{Krishna3}, the natural map
\[
\Omega^*_T(E) \otimes_{C(T)} \bL \to \Omega^*(E)
\]
is an isomorphism, which is an $\bL$-algebra isomorphism if $X$ is smooth.
On the other hand, it follows from Theorem~\ref{thm:Basic} that
the term on the left is same as $\Omega^*(E/B) \otimes_{C(T)} \bL$.
The corollary now follows from Theorem~\ref{thm:Cob-Main}.
\end{proof}
\begin{cor}\label{cor:Cob-Group*}
Let $G$ be a connected algebraic group (not necessarily linear) over $k$ and
let $G \xrightarrow{\alpha_G} A(G)$ be the albanese morphism of $G$.
Then there is an $\bL$-algebra isomorphism
\[
\Omega^*(G) \cong \frac{\Omega^*(A(G))}{I\Omega^*(A(G))}.
\]
\end{cor}
\begin{proof}
It follows immediately from Corollary~\ref{cor:Cob-Group} and the fundamental
exact sequence
\[
1 \to G_{\rm aff} \to G \xrightarrow{\alpha_G} A(G) \to 1,
\]
where $G_{\rm aff}$ is the largest connected linear algebraic subgroup of $G$.
\end{proof}

\section{Cobordism of $E/P$}\label{section:Parabolic}
In this section, we complete the proof of Theorem~\ref{thm:Cob-PR} by 
deducing it from Theorem~\ref{thm:Cob-Main}. This is done by using
the following general technique, which the author learned from an unpublished
note \cite{EL} of Edidin and Larsen. 
So let $(G,T,B,W)$ be our given datum, where we have already reduced to the
case when $G$ is reductive. Since $T$ is split, the group $G$ is given by
its root system $\Phi(G,T)$. Let $\Delta$ be a base of $\Phi$ and $B$
the corresponding Borel. By the well known theory of root system 
({\sl cf.} \cite{Springer}), for every subset $I$ of $\Delta$, there exists
a corresponding parabolic subgroup $P_I \supset B$ and every parabolic
subgroup containing $B$ is of this form. Let $\Psi_I$ be the intersection of
$\Phi$ with the span of $I$ and let $P_I = M_IN_I$ be the corresponding
Levi decomposition, where $M_I$ contains $T$ and $\Phi(M_I, T) = \Psi_I$.
The Weyl group of $M_I$ with respect to $T$ is the subgroup of $W$ generated
by the simple reflections corresponding to the elements of $I$. The Borel
subgroup $B_I$ of $M_I$ corresponding to the positive roots in $\Psi_I$
is the intersection of $B$ with $M_I$ and hence the natural map $P \to M_I$
gives an isomorphism ${P_I}/B \xrightarrow{\cong} {M_I}/{B_I}$. Since
every Borel subgroup is conjugate to $B$, we can assume without loss of
generality that $P_I$ is same as $P$. We then obtain tower of fibrations
\begin{equation}\label{eqn:Tower}
E \to E/N \to E/B \xrightarrow{g} E/P \xrightarrow{f} X.
\end{equation}
We need the following Corollary of Theorem~\ref{thm:Cob-Main} to prove the
case of parabolic flag bundles.
\begin{cor}\label{cor:Weyl-PR}
Let $(G,T,B,W)$ be the datum as above and let $E \to X$ be a principal
$G$-bundle. The natural map $\Omega^*(X) \to {\Omega^*(E/B)}^W$ is an
isomorphism.
\end{cor}
\begin{proof} This follows immediately from Theorem~\ref{thm:Cob-Main} using
the fact that the trivial $W$-module $\Q$ is a projective $\Q[W]$-module,
and $M^W = Hom_{C[W]}\left(\Q, M\right)$ for any $\Q[W]$-module $M$.
\end{proof}
To prove Theorem~\ref{thm:Cob-PR}, we notice that $E/B \to E/P$ is a
$M/B$-bundle, where $B$ is a Borel in $M$ containing $T$. If $W_P$ is Weyl
group of $M$ with respect to $T$, we obtain
\begin{equation}\label{eqn:Weyl-PR1}
\Omega^*(E/P) \otimes_{{C(T)}^{W_P}} C(T) \xrightarrow{\cong} \Omega^*(E/B)
\end{equation}
by Theorem~\ref{thm:Cob-Main}. Taking the $W_P$-invariants and using
Corollary~\ref{cor:Weyl-PR}, we get
\[
\Omega^*(E/P) \xrightarrow{\cong} {\Omega^*(E/B)}^{W_P} \cong
\Omega^*(X) \otimes_{{C(T)}^W} {C(T)}^{W_P}.
\]
This completes the proof of Theorem~\ref{thm:Cob-PR}.
$\hspace*{7cm}\hfil \square$

\section{Higher Chow groups of flag-bundles}
Let $X$ be a scheme and let $p : E \to X$ be a principal $G$-bundle and
let $\pi : E/B \to X$ be the flag bundle associated to a Borel subgroup
of $G$. In this section, we describe the higher Chow groups of $E$ and
$E/B$ in terms of the higher Chow groups of $X$ and the characteristic
classes of the maximal torus $T$. From this, we obtain formula
for the higher Chow groups of the flag bundles $\pi : E/P \to X$.
 
Recall from \cite{Bloch} that the higher
Chow groups of $X$ are given by the homology groups $\ch^i(X, n) =
\h_n\left(Z^i(X, \bullet)\right)$ of the cycle complex $Z^i(X, \bullet)$
of codimension $i$ cycles in the simplicial spaces $X \times \Delta^{\bullet}$.
We refer to {\sl loc. cit.} for more detail and for some standard functorial
properties. For a scheme $X$ with $G$-action, the equivariant higher
Chow groups $\ch^*_G(X, \bullet)$ are defined in \cite{EG}. We refer the
reader to \cite{Krishna1} for more details about these groups.
We denote by $\ch^*(X)$, the full Chow groups $\stackrel{\infty}
{\underset{n = 0}\oplus} \ch^*(X,n)$ of $X$. This is a $\ch^*(k)$-algebra
if $X$ is smooth.  

Let $S = \ch^*_T(k) \cong \Q[x_1, \cdots , x_r]$ and 
$S^W = S(G) = \ch^*_G(k) \cong \Q[\sigma_1, \cdots ,
\sigma_r]$ be as before. Then the same construction as in 
Section~\ref{section:prelim} (or as in \cite{Vistoli}) gives 
a natural map of $S$-modules
\begin{equation}\label{eqn:HCG1}
\alpha_X : \ch^*(X) \otimes_{S^W} S \to \ch^*(E/B).
\end{equation}
\begin{lem}\label{lem:HC-trivial}
Let $E/B \times X \xrightarrow{\pi} X$ be the trivial flag bundle. Then
$\alpha_X$ is an isomorphism.
\end{lem}
\begin{proof}
We first assume $X = {\rm Spec}(k)$. By \cite[Theorem~1.6]{Krishna2},
the natural map $\ch^*_G(G/B) \otimes_{S^W} \Q \to \ch^*(G/B)$
is an isomorphism since $G/B$ is smooth and projective.
On the other hand, we have $\ch^*_G(G/B) \cong \ch^*_T(k)$ by
\cite[Corollary~3.2]{Krishna1}. In particular, we get
\[
\begin{array}{lll}
\ch^*(k) \otimes_{S^W} S & \cong & \left(\ch^*_T(k) \otimes_S \Q\right) 
\otimes_{S^W} S \\
& \cong & \ch^*_T(k) \otimes_S \left(S \otimes_{S^W} \Q\right) \\
& \cong & \ch^*_T(k) \otimes_{S^W} \Q \\
& \cong & \ch^*(G/B),
\end{array}
\]
where the first isomorphism holds by \cite[Theorem~1.6]{Krishna2}.
If $G/B \times X \xrightarrow{\pi} X$ is the trivial bundle, then
we have 
\[
\begin{array}{lll}
\ch^*(X) \otimes_{S^W} S & \cong & \left(\ch^*(X) \otimes_{\ch^*(k)} \ch^*(k)
\right) \otimes_{S^W} S \\
& \cong & \ch^*(X) \otimes_{\ch^*(k)} \left(\ch^*(k) \otimes_{S^W} S\right) \\
& \cong & \ch^*(X) \otimes_{\ch^*(k)} \ch^*(G/B) \\
& \cong & \ch^*(X \times G/B),
\end{array}
\]
where the third isomorphism follows from the case of $X = {\rm Spec}(k)$
and the last isomorphism follows from \cite[Lemma~3.6]{Krishna1}.
This completes the proof.
\end{proof}
\noindent
{\bf {Proof of Theorem~\ref{thm:HC-Main}:}} 
As in the case of cobordism, we can deduce the case of parabolic subgroups
from the case of Borel subgroups. So we prove the result for the flag bundles
of the type $E/B \to X$.
We prove by induction on the dimension of $X$. If $X$ is zero-dimensional,
this follows from the case of $X = {\rm Spec}(k)$ shown above and the
analogue of Corollary~\ref{cor:Retract} for the higher Chow groups.
In general, following the proof in the cobordism case, 
we can find a dense open subset $U \stackrel{j}{\inj} X$ and a finite \'etale
cover $g : U' \to U$ such that the pull-back of $E/B$ is the trivial flag
bundle on $U'$ and the result holds for $U'$ as shown above. We deduce the
result for $U$ by the higher Chow groups analogue of 
Corollary~\ref{cor:Retract}.
Let $Z$ be the complement of $U$ in $X$ with the reduced closed subscheme
structure. We get the following diagram of long exact localization sequences
\[
\xymatrix@C.3pc{
\ch^*(U) \otimes S \ar[r]^{\partial} \ar[d]^{\alpha_U} &
\ch^*(Z) \otimes S \ar[r]^{i_*} \ar[d]^{\alpha_Z} &
\ch^*(X) \otimes S \ar[r]^{j^*} \ar[d]^{\alpha_X} &
\ch^*(U) \otimes S \ar[r]^{\partial} \ar[d]^{\alpha_U} &
\ch^*(Z) \otimes S \ar[d]^{\alpha_Z} \\
\ch^*({E_U}/B) \ar[r]_{{\partial}'} & \ch^*({E_Z}/B) \ar[r]_{i'_*} &
\ch^*(E/B) \ar[r]_{j'^*} & \ch^*({E_U}/B) \ar[r]_{{\partial}'} & 
\ch^*({E_Z}/B),}
\]
where the tensor product in the top row is over the ring $S^W$.
In particular, this row is exact by the flatness of $S$ over $S^W$.
Since the dimension of $Z$ is strictly smaller than that of $X$, we see that
the maps $\alpha_Z$ are isomorphisms by induction. 
We have shown above that the maps
$\alpha_U$ are also isomorphisms. We conclude that $\alpha_X$ is an
isomorphism.
$\hspace*{8cm} \hfil \square$ 
\begin{cor}\label{cor:HC-Main-group}
Let $G$ be a connected linear algebraic group with a split maximal torus $T$.
Let $E \xrightarrow{p} X$ be a principal $G$-bundle over a scheme $X$. Then
\[
\ch^*(E) \xrightarrow{\cong}
\ch^*(X) \otimes_{S(G)} \Q \cong \frac{\ch^*(X)}{I\ch^*(X)},
\]
where $I$ is the ideal of $S(G)$ generated by the Chern classes of
$G$-homogeneous line bundles. This is a $\Q$-algebra isomorphism if $X$ is
smooth.
\end{cor}
\begin{proof}
The proof is exactly same as the proof of Corollary~\ref{cor:Cob-Group}.
\end{proof}
The following corollary recovers a result of Brion 
\cite[Proposition~2.8]{Brion} for the ordinary Chow groups $\ch^*(G,0)$
of connected algebraic groups as a special case.
\begin{cor}\label{cor:HC-Main-Group*}
Let $G$ be a connected algebraic group (not necessarily linear) over $k$ and
let $G \xrightarrow{\alpha_G} A(G)$ be the albanese morphism of $G$.
Then there is a $\Q$-algebra isomorphism
\[
\ch^*(G) \cong \frac{\ch^*(A(G))}{I\ch^*(A(G))}.
\]
\end{cor}
\begin{proof}
The proof is exactly same as the proof of Corollary~\ref{cor:Cob-Group*}.
\end{proof}
%\begin{ack}
\noindent\emph{Acknowledgments.}
This work was done while the author was visiting the Institute Fourier,
Universit\'e de Grenoble in June, 2010. He would like to thank Michel 
Brion for the invitation and the local financial support during his 
stay.
%\end{ack} 

\end{document}